\documentclass[twoside,english,american,DIV12,listtotoc,bibtotoc,idxtotoc]{scrartcl}
\usepackage[T1]{fontenc}
\usepackage[latin1]{inputenc}
\usepackage{verbatim}
\usepackage{amsthm}
\usepackage{amsmath}
\usepackage{amssymb}

\makeatletter
\theoremstyle{plain}
\newtheorem{thm}{\protect\theoremname}[section]
  \theoremstyle{plain}
  \newtheorem{lem}[thm]{\protect\lemmaname}
  \theoremstyle{definition}
  \newtheorem{defn}[thm]{\protect\definitionname}
  \theoremstyle{plain}
  \newtheorem{prop}[thm]{\protect\propositionname}
  \theoremstyle{definition}
  \newtheorem{example}[thm]{\protect\examplename}
  \theoremstyle{remark}
  \newtheorem{rem}[thm]{\protect\remarkname}
  \theoremstyle{definition}
  \newtheorem{condition}[thm]{\protect\conditionname}

\usepackage{helvet}
\usepackage[T1]{fontenc}

\usepackage{amsmath}
\usepackage{setspace}
\usepackage{amssymb}
\usepackage{enumerate}
\usepackage{url}

\makeatletter

\usepackage[T1]{fontenc}    

\usepackage{a4wide}        
\usepackage{tu-preprint}
\usepackage{amsmath}

\usepackage{euscript}
\usepackage{mathtools}
\allowdisplaybreaks[1] 
\usepackage{amsfonts}
\newenvironment{keywords}{ \noindent\footnotesize\textbf{Keywords and phrases:}}{}

\newenvironment{class}{\noindent\footnotesize\textbf{Mathematics subject classification 2010:}}{}

\usepackage{color,tu-preprint}

\newcommand*{\dive}{\operatorname{div}}

\newcommand*{\Grad}{\operatorname{Grad}}
\newcommand*{\Div}{\operatorname{Div}}

\newcommand*{\grad}{\operatorname{grad}}

\renewcommand*{\i}{\mathrm{i}}

\DeclareMathAccent{\Circ}{\mathalpha}{operators}{"17}
\newcommand{\interior}[1]{\Circ{#1}}
\renewcommand{\Im}{\operatorname{\mathfrak{Im}}}
\renewcommand{\Re}{\operatorname{\mathfrak{Re}}}

\newcommand{\oi}[2]{\left]#1,#2 \right[}
\newcommand{\rga}[1]{\left]#1,\infty  \right[}

\newcommand{\lci}[2]{\left[#1,#2 \right[}
\newcommand{\rci}[2]{\left]#1,#2 \right]}

\renewcommand{\tilde}{\widetilde}
\renewcommand*{\epsilon}{\varepsilon}
\renewcommand*{\theta}{\vartheta}
\renewcommand*{\rho}{\varrho}
\makeatother

\usepackage{babel}
\institut{Institut f\"ur Analysis}

\preprintnumber{MATH-AN-05-2013}

\preprinttitle{On Evolutionary Equations with Material Laws Containing Fractional Integrals.}

\author{Rainer Picard, Sascha Trostorff \& Marcus Waurick}

\AtBeginDocument{
  
}

\makeatother

\usepackage{babel}
  \addto\captionsamerican{\renewcommand{\conditionname}{Condition}}
  \addto\captionsamerican{\renewcommand{\definitionname}{Definition}}
  \addto\captionsamerican{\renewcommand{\examplename}{Example}}
  \addto\captionsamerican{\renewcommand{\lemmaname}{Lemma}}
  \addto\captionsamerican{\renewcommand{\propositionname}{Proposition}}
  \addto\captionsamerican{\renewcommand{\remarkname}{Remark}}
  \addto\captionsamerican{\renewcommand{\theoremname}{Theorem}}
  \addto\captionsenglish{\renewcommand{\conditionname}{Condition}}
  \addto\captionsenglish{\renewcommand{\definitionname}{Definition}}
  \addto\captionsenglish{\renewcommand{\examplename}{Example}}
  \addto\captionsenglish{\renewcommand{\lemmaname}{Lemma}}
  \addto\captionsenglish{\renewcommand{\propositionname}{Proposition}}
  \addto\captionsenglish{\renewcommand{\remarkname}{Remark}}
  \addto\captionsenglish{\renewcommand{\theoremname}{Theorem}}
  \providecommand{\conditionname}{Condition}
  \providecommand{\definitionname}{Definition}
  \providecommand{\examplename}{Example}
  \providecommand{\lemmaname}{Lemma}
  \providecommand{\propositionname}{Proposition}
  \providecommand{\remarkname}{Remark}
\providecommand{\theoremname}{Theorem}

\begin{document}
\selectlanguage{english}%
\makepreprinttitlepage

\author{ Rainer Picard, Sascha Trostorff \& Marcus Waurick\\ Institut f\"ur Analysis, Fachrichtung Mathematik\\ Technische Universit\"at Dresden\\ Germany\\ rainer.picard@tu-dresden.de\\ sascha.trostorff@tu-dresden.de\\ marcus.waurick@tu-dresden.de} 

\selectlanguage{american}%

\title{On Evolutionary Equations with Material Laws Containing Fractional
Integrals.}

\maketitle
\begin{abstract} \textbf{Abstract.} A well-posedness result for a
time-shift invariant class of evolutionary operator equations involving
material laws with fractional time-integrals of order $\alpha\in\oi01$
is considered and exemplified by an application to a Kelvin-Voigt
type model.\end{abstract}

\begin{keywords} fractional derivatives, fractional integrals, evolutionary
equations, visco-elasticity, Kelvin-Voigt model, fractional Fokker-Planck
equation, causality \end{keywords}

\begin{class} 35A01 (existence of solutions to PDEs), 35A02 (uniqueness
of solutions to PDEs), 35A22 (transform methods for PDEs), 26A33 (fractional
derivatives and integrals),74C10 (visco-elasticity), 74H99 (dynamical
problems in mechanics of deformable solids) \foreignlanguage{english}{ \end{class}}

\selectlanguage{english}%
\newpage

\tableofcontents{} 

\newpage

\selectlanguage{american}%

\section{Introduction}

Following leading examples of mathematical physics, abstract linear
evolutionary problems of the form 
\begin{equation}
\partial_{0}V+AU=f\,\mbox{ on }\mathbb{R}\label{eq:evo1}
\end{equation}
come into focus. Here $\partial_{0}$ denotes the operator of time-differentiation,
established as a normal, boundedly invertible operator in a suitable
Hilbert space (see \cite{Picard1989}), and we assume that $A$ is
a closed, densely defined operator, such that $A+\lambda_{0}$ and
$A^{*}+\lambda_{0}$ are both maximal accretive for some $\lambda_{0}\in\mathbb{R}$
in some Hilbert space setting. To determine $V$ and $U$ from the
given data $f$ the equation (\ref{eq:evo1}) needs of course to be
completed, here by an additional rule of the form
\begin{equation}
V=\mathcal{M}U\label{eq:material_E}
\end{equation}
frequently referred to as a ``material law''. Here $\mathcal{M}$
is given in terms of an operator-valued function calculus associated
with the (inverse) time-derivative via a bounded-operator-valued and
analytic function $M$. The theory is set up in \cite{PDE_DeGruyter}
for the case $A$ skew-selfadjoint and in \cite{OP2012,Trostorff2012_nonlin_bd}
for the general case resulting in well-posedness of the evolutionary
problem (we omit the closure bar over the operator sum)
\begin{equation}
\left(\partial_{0}M\left(\partial_{0}^{-1}\right)+A\right)U=f\mbox{ on \ensuremath{\mathbb{R}.}}\label{eq:evo_2}
\end{equation}
In order for this to describe a process evolving in time we not just
want the existence of a solution operator $\left(\partial_{0}M\left(\partial_{0}^{-1}\right)+A\right)^{-1}$
but we would also expect a causality property to hold. Causality,
as the property that solutions vanish as long as data are zero, can
be conveniently encoded as
\begin{equation}
\chi_{_{\oi{-\infty}a}}\left(m_{0}\right)\left(\partial_{0}M\left(\partial_{0}^{-1}\right)+A\right)^{-1}=\chi_{_{\oi{-\infty}a}}\left(m_{0}\right)\left(\partial_{0}M\left(\partial_{0}^{-1}\right)+A\right)^{-1}\chi_{_{\oi{-\infty}a}}\left(m_{0}\right)\label{eq:causal}
\end{equation}
for all $a\in\mathbb{R}$. Here $\chi_{_{\oi{-\infty}a}}\left(m_{0}\right)$
denotes the temporal cut-off operator
\[
\left(\chi_{_{\oi{-\infty}a}}\left(m_{0}\right)\varphi\right)\left(t\right)=\chi_{_{\oi{-\infty}a}}\left(t\right)\:\varphi\left(t\right)\mbox{ for }t\in\mathbb{R}.
\]
The symbol $m_{0}$ serves as a reminder for 'multiplication-by-argument'
with respect to the time parameter. Since the systems considered are
time-shift invariant it would suffice to require (\ref{eq:causal})
just for one $a\in\mathbb{R}$ -- say $a=0$. 

In this approach well-posedness is warranted by the eventual (real)
strict positive definiteness of $\partial_{0}M\left(\partial_{0}^{-1}\right)$
for all sufficiently large $\rho\in\oi0\infty$. In this paper we
want to consider material laws of the specific form
\begin{equation}
M\left(\partial_{0}^{-1}\right)=M_{0}+\sum_{\alpha\in\Pi}\partial_{0}^{-\alpha}M_{\alpha}+\partial_{0}^{-1}M_{1}+\partial_{0}^{-1}M_{2}\left(\partial_{0}^{-1}\right),\label{eq:frac-law}
\end{equation}
where $\Pi\subseteq\oi01\mbox{ finite, }$$M_{\alpha}\mbox{ continuous and selfadjoint for }\alpha\in\Pi.$
The terms $\partial_{0}^{-\alpha}$ are the fractional integrals of
the title, which can be defined in terms of the said function calculus.
Such material laws have been of interest in many applications for
a long time and although the strict positive definiteness requirement
is exhaustive, it is not clear what manageable conditions on the coefficient
operators $M_{0}$, $M_{\alpha}$, $\alpha\in\Pi$, would result in
a corresponding positivity estimate, see (\ref{eq:solution-condition})
below. It is the purpose of this paper to give such conditions. For
this we first need to establish the said function calculus associated
with the time derivative $\partial_{0}$ allowing in particular for
a proper definition of the fractional integral $\partial_{0}^{-\alpha}$,
$\alpha\in\oi01$. This is done in the Section \ref{sec:Fractional-Calculus-and}.
Although there exists a vast literature on the topic of fractional
derivatives (see for example the monographs \cite{citeulike:1428674,0789.26002}
and the references therein), the article is largely self-contained,
so that there is no much need to refer to classical statements on
fractional derivatives to understand the results in this work. To
the best of the authors' knowledge it has not been widely noted that
$\partial_{0}$ can be established as a normal operator, see \cite{Picard1989},
which then comes with its own standard function calculus via the spectral
theorem for such operators. Indeed, it turns out that a unitary variant
of the Fourier-Laplace transformation yields a spectral representation
for $\Im\partial_{0}=\frac{1}{2\i}\overline{\left(\partial_{0}-\partial_{0}^{*}\right)}$
and $\Re\partial_{0}$ is just multiplication by a number. This simplifies
matters considerably and since we are staying in a Hilbert space setting
there is indeed no need to utilize more intricate results of fractional
calculus in other spaces. This approach was already successfully used
to study homogenization in fractional elasticity, \cite{waurick2013_fract_elast}.
After establishing the function calculus we shall characterize in
Section \ref{sec:A-Class-of} a class of material laws for which (\ref{eq:solution-condition})
can be shown to hold. This is the main result of the paper. Compared
to the results in \cite[Theorems 2.1 and 2.2]{waurick2013_fract_elast}
and \cite{PicardPamm} we will find that our more in-depth study here
implies less restrictive assumptions on the class of material laws.
The final Section \ref{sec:Some-Applications} is devoted to illustrating
the general results with applications to a fractional Fokker-Planck
equation (Section \ref{sub:Fractional-Reaction-Diffusion-Eq}) and
a fractional Kelvin-Voigt type model from solid mechanics (Section
\ref{sec:An-Application-to}). Section \ref{sec:Some-Applications}
is concluded with some remarks of how to deal with initial value problems
(Section \ref{sub:Initial-Boundary-Value}).

\section{Functional analytic framework}

\subsection{\label{sec:Fractional-Calculus-and}Fractional Calculus and Operator-Valued
Functions of Time Differentiation}

As indicated in the introduction, we start by establishing time differentiation
$\partial_{0}$ as a normal operator. We consider the weighted $H$-valued
$L^{2}$-type space $H_{\rho,0}\left(\mathbb{R},H\right)$, generated
by completion of $\interior C_{\infty}\left(\mathbb{R},H\right)$,
the space of smooth $H$-valued functions with compact support, with
respect to the inner product $\left\langle \:\cdot\:|\:\cdot\:\right\rangle _{\rho,0,0}$
given by the weighted integral $\left(\varphi,\psi\right)\mapsto\int_{\mathbb{R}}\left\langle \varphi\left(t\right)|\psi\left(t\right)\right\rangle _{H}\:\exp\left(-2\rho t\right)dt.$
The associated norm will be denoted by $\left|\:\cdot\:\right|_{\rho,0,0}$.
The multiplication operator $\interior C_{\infty}\left(\mathbb{R},H\right)\subseteq H_{\rho,0}\left(\mathbb{R},H\right)\to\interior C_{\infty}\left(\mathbb{R},H\right)\subseteq H_{0,0}\left(\mathbb{R},H\right)=L^{2}\left(\mathbb{R},H\right),\:\varphi\mapsto\left(t\mapsto\exp\left(-\rho t\right)\varphi\left(t\right)\right)$,
clearly has a unitary continuous extension, which we shall denote
by $\exp\left(-\rho m_{0}\right)$. Its inverse (adjoint) is given
by $\exp\left(\rho m_{0}\right)$. By taking the closure of the operator
\begin{align*}
\interior C_{\infty}(\mathbb{R},H)\subseteq H_{\rho,0}(\mathbb{R},H) & \to H_{\rho,0}(\mathbb{R},H)\\
\phi & \mapsto\phi',
\end{align*}
the time-derivative $\partial_{0}$ can be established as a normal
operator on $H_{\rho,0}\left(\mathbb{R},H\right)$ with real part
$\rho$ and so with $\frac{1}{\i}\left(\partial_{0}-\rho\right)$
as imaginary part. The domain of $\partial_{0}$ can be characterized
by functions belonging to $H_{\rho,0}(\mathbb{R},H),$ whose weak
derivatives also lie in $H_{\rho,0}(\mathbb{R},H).$ For $\rho\in\mathbb{R}\setminus\left\{ 0\right\} $
we have the bounded invertibility of $\partial_{0}$. The inverse
of the normal operator $\partial_{0}$ is bounded and can be described
by
\begin{align*}
\left(\partial_{0}^{-1}f\right)\left(t\right) & =\begin{cases}
\int_{-\infty}^{t}f\left(s\right)\, ds & \mbox{ if }\rho>0,\\
-\intop_{t}^{\infty}f(s)\, ds & \mbox{ if }\rho<0
\end{cases}
\end{align*}
for every $f\in H_{\rho,0}\left(\mathbb{R},H\right)$ and almost every
$t\in\mathbb{R}$ as a Bochner integral. Henceforth we shall focus
on the case $\rho\in\oi0\infty$, which is associated with (forward)
causality, see e.g.~equation (\ref{eq:causal}). The Fourier-Laplace
transform $\mathcal{L}_{\rho}\;:=\mathcal{F}\,\exp\left(-\rho m_{0}\right):H_{\rho,0}(\mathbb{R},H)\to L^{2}(\mathbb{R},H)$,
given as a composition of the (temporal) Fourier transform $\mathcal{F}$
and the unitary weight operator $\exp\left(-\rho m_{0}\right)$, is
a spectral representation associated with $\partial_{0}$. It is 
\[
\partial_{0}=\mathcal{L}_{\rho}^{*}\left(\mathrm{i}\, m_{0}+\rho\right)\:\mathcal{L}_{\rho},
\]
where $m_{0}$ denotes the multiplication-by-argument operator given
as the closure of
\begin{eqnarray*}
\interior C_{\infty}\left(\mathbb{R},H\right)\subseteq L^{2}\left(\mathbb{R},H\right) & \to & L^{2}\left(\mathbb{R},H\right)\\
\varphi & \mapsto & m_{0}\varphi
\end{eqnarray*}
with
\[
\left(m_{0}\varphi\right)\left(\lambda\right)\::=\lambda\:\varphi\left(\lambda\right)\:\mbox{in }\, H
\]
for every $\lambda\in\mathbb{R}$. This observation allows us to consistently
define an operator function calculus associated with $\partial_{0}$
in a standard way, \cite{Dunford_Schwartz1988}. Clearly, we can even
extend this calculus to operator-valued functions by letting 
\[
M\left(\partial_{0}^{-1}\right)\;\coloneqq\mathcal{L}_{\rho}^{*}M\left(\frac{1}{\i\: m_{0}+\rho}\right)\mathcal{L}_{\rho}.
\]
Here the linear operator $M\left(\frac{1}{\i\: m_{0}+\rho}\right):L^{2}\left(\mathbb{R},H\right)\to L^{2}\left(\mathbb{R},H\right)$
is  determined uniquely via
\[
\left(M\left(\frac{1}{\i\: m_{0}+\rho}\right)\varphi\right)\left(\lambda\right)\:\coloneqq M\left(\frac{1}{\i\:\lambda+\rho}\right)\,\varphi\left(\lambda\right)\:\mbox{ in }H
\]
for every $\lambda\in\mathbb{R}$,~ $\varphi\in\interior C_{\infty}\left(\mathbb{R},H\right)$
by an operator-valued function $M$. For a material law the operator-valued
function $M$ needs to be bounded and an analytic function $z\mapsto M\left(z\right)$
in an open ball $B_{\mathbb{C}}\left(r,r\right)$ with some positive
radius $r$ centered at $r$. This is not an artificial assumption,
rather a necessary constraint enforced by the requirement of causality,
see \cite{PDE_DeGruyter} or \cite[Theorem 9.1]{Thomas1997}. 

In terms of the associated operator-valued function calculus we also
know what
\[
\partial_{0}^{-\alpha},\:\alpha\in[0,1[,
\]
means%
\footnote{It should be noted that $\partial_{0}^{-\alpha}$ is largely independent
of the particular choice of $\rho\in\rga0$. Indeed, since 
\[
\mathcal{L}_{\rho}\chi_{_{\rga0}}\left(m_{0}\right)m_{0}^{\alpha-1}=\frac{1}{\sqrt{2\pi}}\frac{\Gamma\left(\alpha\right)}{\left(\i m+\rho\right)^{\alpha}}
\]
we have for $\varphi\in\interior C_{\infty}\left(\mathbb{R},H\right)$
\[
\frac{1}{\Gamma\left(\alpha\right)}\chi_{_{\rga0}}\left(m_{0}\right)m_{0}^{\alpha-1}*\varphi=\partial_{0}^{-\alpha}\varphi
\]
and
\[
\left(\frac{1}{\Gamma\left(\alpha\right)}\chi_{_{\rga0}}\left(m_{0}\right)m_{0}^{\alpha-1}*\varphi\right)\left(t\right)=\frac{1}{\Gamma\left(\alpha\right)}\int_{-\infty}^{t}\frac{1}{\left(t-s\right)^{1-\alpha}}\varphi\left(s\right)\: ds.
\]
 From this convolution integral representation we can also read off
that $\partial_{0}^{-\alpha}$ is causal.%
}. With this we define for $\gamma\in\mathbb{R}$
\begin{equation}
\partial_{0}^{\gamma}\coloneqq\partial_{0}^{\left\lceil \gamma\right\rceil }\partial_{0}^{\gamma-\left\lceil \gamma\right\rceil }\label{frac-dif}
\end{equation}
as a natural generalization of differentiation to arbitrary real orders.
Here $\left\lceil \alpha\right\rceil $ denotes the smallest integer
greater or equal to $\alpha$. The fact that $\left(\partial_{0}^{\gamma}\right)_{\gamma\in\mathbb{R}}$
is a family of commuting operators appears to be useful in applications,
see \cite{doi:10.1142/S0218396X03002024} and the quoted literature.
In contrast, the variant%
\footnote{In the limit $a\to-\infty$ the spectral fractional derivative is
formally recovered:
\[
\partial_{0}^{\gamma}={}_{-\infty}D_{t}^{\gamma}.
\]
} 
\[
_{a}D_{t}^{\gamma}\coloneqq\partial_{0}^{\gamma}\chi_{_{\oi{a}\infty}}\left(m_{0}\right)=\partial_{0}^{\left\lceil \gamma\right\rceil }\left(\partial_{0}^{\gamma-\left\lceil \gamma\right\rceil }\chi_{_{\oi{a}\infty}}\left(m_{0}\right)\right)
\]
with an appropriate choice of domain, known as Riemann-Liouville fractional
derivative, $a\in\mathbb{R}$ a parameter, \cite{citeulike:1428674,0789.26002},
lacks this property. This is also true for the frequently used alternative
fractional derivative, the Caputo fractional derivative, \cite{citeulike:1428674,0789.26002}:
\[
_{a}^{C}D_{t}^{\gamma}\coloneqq\partial_{0}^{\gamma-\left\lceil \gamma\right\rceil }\chi_{_{\oi{a}\infty}}\left(m_{0}\right)\partial_{0}^{\left\lceil \gamma\right\rceil }
\]
with a suitable domain. With these fractional derivatives being mere
variants of limited usefulness in our context, we shall utilize only
the above spectral definition (\ref{frac-dif}) for our fractional
differentiation/integration. 

Let us inspect more closely some properties of $\partial_{0}^{\alpha}$
for $\alpha\in\oi01$. Denoting by $\left\Vert \,\,\cdot\:\right\Vert _{X}$
the operator norm for operators in the space $X$ we record our first
lemma.
\begin{lem}
\label{lem-frac-est-1}For $\alpha\in\left[0,1\right]$ we have
\begin{equation}
\Re\partial_{0}^{\alpha}\geq\rho^{\alpha}\label{eq:alphalow}
\end{equation}
and for $\alpha\in\lci0\infty$
\begin{equation}
\left\Vert \partial_{0}^{-\alpha}\right\Vert _{H_{\rho,0}\left(\mathbb{R},H\right)}\leq\rho^{-\alpha}.\label{eq:alphahigh}
\end{equation}
\end{lem}
\begin{proof}
To see (\ref{eq:alphalow}) it suffices -- by the function calculus
of $\partial_{0}$ -- to consider $\Re\left(\i t+\rho\right)^{\alpha}$
for $t\in\mathbb{R}.$ We have
\begin{align}
\Re\left(\left(\i t+\rho\right)^{\alpha}\right) & =\left|\i t+\rho\right|^{\alpha}\cos\left(\alpha\beta\right)\nonumber \\
 & =\rho^{\alpha}\left(1+\left(\tan\beta\right)^{2}\right)^{\alpha/2}\cos\left(\alpha\beta\right)\nonumber \\
 & =\rho^{\alpha}\frac{\cos\left(\alpha\beta\right)}{\left(\cos\beta\right)^{\alpha}}\label{eq:real_part_frac}
\end{align}
with
\[
\beta\coloneqq\arg\left(\i t+\rho\right)\in]-\pi/2,\pi/2[\,.
\]
The values of $\frac{\cos\left(\alpha\beta\right)}{\left(\cos\beta\right)^{\alpha}}$
are either unbounded%
\footnote{\label{fn:incomparable}This fact makes the cases $\alpha\in\left\{ 0,1\right\} $
and $\alpha\in\oi01$ incomparable.%
} as $\beta\to\pm\frac{\pi}{2}$ ($\alpha\in]0,1[$), or $1$ (if $\alpha\in\{0,1\}$).
In the first case we shall try to find a minimum. The derivative
\begin{align*}
\left(\beta\mapsto\frac{\cos\left(\alpha\beta\right)}{\left(\cos\beta\right)^{\alpha}}\right)^{\prime} & =\left(\beta\mapsto-\frac{\alpha\left(\cos\beta\right)^{\alpha}\sin\left(\alpha\beta\right)-\cos\left(\alpha\beta\right)\alpha\left(\cos\left(\beta\right)\right)^{\alpha-1}\sin\left(\beta\right)}{\left(\cos\beta\right)^{2\alpha}}\right)\\
 & =\left(\beta\mapsto-\alpha\frac{\cos\beta\sin\left(\alpha\beta\right)-\cos\left(\alpha\beta\right)\sin\left(\beta\right)}{\left(\cos\beta\right)^{\alpha+1}}\right)\\
 & =\left(\beta\mapsto-\alpha\frac{\sin\left(\left(\alpha-1\right)\beta\right)}{\left(\cos\beta\right)^{\alpha+1}}\right)
\end{align*}
vanishes only for $\beta=0.$ Due to the behavior at the limits this
must be the point where a minimum occurs. For $\beta=0$ we have,
however, 
\[
\frac{\cos\left(\alpha\beta\right)}{\left(\cos\beta\right)^{\alpha}}=1,
\]
which proves (\ref{eq:alphalow}). Estimate (\ref{eq:alphahigh})
now follows for $\alpha\in\left[0,1\right]$ since
\[
\left|U\right|_{\rho,0,0}\left|\partial_{0}^{\alpha}U\right|_{\rho,0,0}\geq\Re\left\langle U|\partial_{0}^{\alpha}U\right\rangle _{\rho,0,0}\geq\rho^{\alpha}\left\langle U|U\right\rangle _{\rho,0,0}
\]
 and so
\[
\left|\partial_{0}^{\alpha}U\right|_{\rho,0,0}\geq\rho^{\alpha}\left|U\right|_{\rho,0,0}
\]
for all $U\in D\left(\partial_{0}^{\alpha}\right)$, which implies
(\ref{eq:alphahigh}) for $\alpha\in\left[0,1\right]$. Since
\[
\partial_{0}^{-\alpha}=\partial_{0}^{-\left\lfloor \alpha\right\rfloor }\partial_{0}^{\left\lfloor \alpha\right\rfloor -\alpha}
\]
with $\left\lfloor \alpha\right\rfloor $ denoting the largest integer
less or equal to $\alpha$, (\ref{eq:alphahigh}) follows for arbitrary
$\alpha\in\lci0\infty$ by the submultiplicativity of the operator
norm and the fact that $\|\partial_{0}^{-1}\|\leq\rho^{-1}$. 
\end{proof}
\noindent For our purposes we mostly need to be concerned with $\alpha\in\oi01$
leading to 
\[
\partial_{0}^{\alpha}=\partial_{0}\partial_{0}^{\alpha-1},\:\alpha\in\oi01.
\]

\begin{lem}
\label{lem-frac-mono} Let $I$ be a compact interval in $]0,1[$.
Then the function
\[
I\ni\alpha\mapsto\Re\left(\partial_{0}^{\alpha}\right),
\]
with unbounded selfadjoint operators $\Re\left(\partial_{0}^{\alpha}\right)$
in $H_{\rho,0}\left(\mathbb{R},H\right)$ as values, is \emph{monotonically
increasing }in the sense that there exists $\rho_{0}>0$ such that
for all $\rho>\rho_{0}$ and for all $\alpha\leq\beta$ in $I$ the
estimate
\[
\left\langle U|\Re\left(\partial_{0}^{\alpha}\right)U\right\rangle _{\rho,0,0}\leq\left\langle U|\Re\left(\partial_{0}^{\beta}\right)U\right\rangle _{\rho,0,0}
\]
holds for all $U\in D\left(\partial_{0}^{\alpha}\right)$.\end{lem}
\begin{proof}
We define 
\[
\rho_{0}\coloneqq\exp\left(\frac{\pi}{2}\tan\left(\max I\:\frac{\pi}{2}\right)\right)>0.
\]
Let $t\in\mathbb{R}$ be fixed and $\rho>\rho_{0}$. We recall from
(\ref{eq:real_part_frac}) that for $\alpha\in[0,1]$ 
\begin{align*}
\Re\left(\left(\i t+\rho\right)^{\alpha}\right) & =\rho^{\alpha}\frac{\cos\left(\alpha\beta\right)}{\left(\cos\beta\right)^{\alpha}}
\end{align*}
with
\[
\beta\coloneqq\arg\left(\i t+\rho\right)\in]-\pi/2,\pi/2[\,.
\]
Consider now
\[
\left(\alpha\mapsto\rho^{\alpha}\frac{\cos\left(\alpha\beta\right)}{\left(\cos\beta\right)^{\alpha}}\right)^{\prime}=\left(\alpha\mapsto\left(\frac{\rho}{\cos\beta}\right)^{\alpha}\cos\left(\alpha\beta\right)\left(\ln\frac{\rho}{\cos\beta}\:-\beta\tan\left(\alpha\beta\right)\right)\right).
\]
Noting that
\begin{align*}
\ln\frac{\rho}{\cos\beta}\:-\beta\tan\left(\alpha\beta\right) & \geq\ln\rho\:-\frac{\pi}{2}\tan\left(\alpha\frac{\pi}{2}\right)\\
 & \geq\ln\rho\:-\frac{\pi}{2}\tan\left(\max I\,\frac{\pi}{2}\right)>0
\end{align*}
for all $\alpha\in I$ (note $0<\max I<1$) we see that for $\alpha\leq\beta$
in $I$ 
\begin{equation}
\Re\left(\left(\i t+\rho\right)^{\alpha}\right)\leq\Re\left(\left(\i t+\rho\right)^{\beta}\right)\label{eq:monot}
\end{equation}
 for all $t\in\mathbb{R}$. Via the function calculus of $\partial_{0}$
this yields the (dense) inclusion 
\[
D\left(\partial_{0}^{\alpha}\right)\subseteq D\left(\partial_{0}^{\beta}\right).
\]
For $U\in D\left(\partial_{0}^{\alpha}\right)$ we also see from (\ref{eq:monot})
that
\begin{align*}
\left\langle U|\Re\left(\partial_{0}^{\alpha}\right)U\right\rangle _{\rho,0,0} & =\left\langle \mathcal{L}_{\rho}U|\Re\left(\left(\i m_{0}+\rho\right)^{\alpha}\right)\mathcal{L}_{\rho}U\right\rangle _{0,0,0}\\
 & \leq\left\langle \mathcal{L}_{\rho}U|\Re\left(\left(\i m_{0}+\rho\right)^{\beta}\right)\mathcal{L}_{\rho}U\right\rangle _{0,0,0}=\left\langle U|\Re\left(\partial_{0}^{\beta}\right)U\right\rangle _{\rho,0,0}.\tag*{\qedhere}
\end{align*}

\end{proof}

\subsection{Sobolev-chains\label{sub:Sobolev-chains}}

In this subsection we recall the notion of Sobolev chains (see \cite[Section 2.1]{PDE_DeGruyter})
associated with a closed, boundedly invertible operator on a Hilbert
space. Throughout let $H$ be a Hilbert space and $C:D(C)\subseteq H\to H$
a densely defined closed linear operator with $0\in\rho(C).$ 
\begin{defn}
For $k\in\mathbb{Z}$ we define the inner product $\langle\cdot|\cdot\rangle_{H_{k}(C)}$
on $D(C^{k})$ by 
\[
\langle x|y\rangle_{H_{k}(C)}\coloneqq\langle C^{k}x|C^{k}y\rangle_{H}\quad(x,y\in H_{k}(C)).
\]
Moreover, we define $H_{k}(C)$ as the completion of $D(C^{k})$ with
respect to the inner product $\langle\cdot|\cdot\rangle_{H_{k}(C)}.$
The sequence $(H_{k}(C))_{k\in\mathbb{Z}}$ is called the \emph{Sobolev-chain}
\emph{associated with $C$.}\end{defn}
\begin{prop}[{\selectlanguage{english}%
\cite[Theorem 2.1.6]{PDE_DeGruyter}\selectlanguage{american}
}]
 For each $j,k\in\mathbb{Z}$ with $j\leq k$ we have $H_{k}(C)\hookrightarrow H_{j}(C),$
where the embedding is continuous and dense. Moreover, the operator
\begin{align*}
D(C^{|k|+1})\subseteq H_{k+1} & \to H_{k}\\
x & \mapsto Cx
\end{align*}
has a unitary extension, which will be again denoted by $C.$\end{prop}
\begin{example}
For $\rho>0$ the operator $\partial_{0}:D(\partial_{0})\subseteq H_{\rho,0}(\mathbb{R},H)\to H_{\rho,0}(\mathbb{R},H)$
is densely defined and closed. Moreover, since $\Re\partial_{0}=\rho$,
we have $0\in\rho(\partial_{0}).$ Thus, we can define the Sobolev-chain
associated with the time-derivative, which will be denoted by 
\[
H_{\rho,k}(\mathbb{R};H)\coloneqq H_{k}(\partial_{0,\rho})\quad(k\in\mathbb{Z}).
\]
Then the Dirac-distribution $\delta$ lies in $H_{\rho,-1}(\mathbb{R},H)$
with $\partial_{0}^{-1}\delta=\chi_{[0,\infty[}\in H_{\rho,0}(\mathbb{R},H).$\end{example}
\begin{rem}
We can extend the operator $C:D(C)\subseteq H\to H$ to the Hilbert
space $H_{\rho,0}(\mathbb{R},H)$ for each $\rho>0$ in the canonical
way, i.e., we set $\left(Cu\right)(t)=C\left(u(t)\right)$ for every
$u\in H_{\rho,0}(\mathbb{R},H)$ taking values in the domain of $C$
and almost every $t\in\mathbb{R}.$ Then, obviously the operators
$\partial_{0}$ and $C$ commute. Thus, for each $k,j\in\mathbb{Z}$
the operators can be realized as unitary operators 
\[
\partial_{0}:H_{\rho,k+1}(\mathbb{R},H_{j}(C))\to H_{\rho,k}(\mathbb{R},H_{j-1}(C))
\]
and 
\[
C:H_{\rho,k}(\mathbb{R},H_{j+1}(C))\to H_{\rho,k}(\mathbb{R},H_{j}(C)).
\]
Moreover, for all $M\in L(H_{\rho,0}(\mathbb{R},H))$ being such that
$\partial_{0}^{-1}M=M\partial_{0}^{-1}$ there is a unique continuous
extension to an operator on $H_{\rho,k}(\mathbb{R},H)$ for all $k\in\mathbb{Z}$.
\end{rem}

\section{\label{sec:A-Class-of}A Class of Fractional Material Laws }

\noindent We are now able to rigorously establish the specific form
of the material law operators we wish to consider, namely
\begin{equation}
M\left(\partial_{0}^{-1}\right)=M_{0}+\sum_{\alpha\in\Pi}\partial_{0}^{-\alpha}M_{\alpha}+\partial_{0}^{-1}M_{1}+\partial_{0}^{-1}M_{2}\left(\partial_{0}^{-1}\right).\label{eq:frac-law-1}
\end{equation}
Here $\Pi\subseteq]0,1[$ is a finite set, $M_{\alpha}$ for $\alpha\in\left\{ 0\right\} \cup\Pi\cup\left\{ 1\right\} $
are bounded linear operators in $H$ and $M_{2}$ is a bounded operator
valued function, analytic in a ball $B_{\mathbb{C}}\left(r,r\right)$
for some $r\in\oi0\infty$ with $r>\frac{1}{2\rho}$, compare \cite{PDE_DeGruyter,OP2012}.
Then the well-posedness of the corresponding evolutionary problem
(\ref{eq:evo_2}) can be shown with the help of the main result presented
in \cite{Pi2009-1}.
\begin{thm}[{\selectlanguage{english}%
\cite[Solution Theory]{Pi2009-1}\selectlanguage{american}
}]
\label{thm:sol_theory_orig}For $r>0$ let $M:B_{\mathbb{C}}(r,r)\to L(H)$
be a bounded analytic mapping satisfying 
\begin{equation}
\bigvee_{c>0}\bigwedge_{z\in B_{\mathbb{C}}(r,r)}\Re z^{-1}M(z)\geq c.\label{eq:sol-cond_1}
\end{equation}
Moreover, let $A:D(A)\subseteq H\to H$ be a skew-selfadjoint operator.
Then, for each $\rho>\frac{1}{2r}$ the operator $\partial_{0}M(\partial_{0}^{-1})+A$
is invertible and the inverse 
\[
\left(\partial_{0}M(\partial_{0}^{-1})+A\right)^{-1}:H_{\rho,0}(\mathbb{R},H)\to H_{\rho,0}(\mathbb{R},H)
\]
is bounded and causal in the sense of (\ref{eq:causal}). Moreover,
for each $k\in\mathbb{Z}$ the solution operator $\left(\partial_{0}M(\partial_{0}^{-1})+A\right)^{-1}$
can be established as a bounded linear operator on $H_{\rho,k}(\mathbb{R},H).$ 
\end{thm}
\noindent In order to apply this well-posedness result, we need to
require the positive definiteness constraint (\ref{eq:sol-cond_1})
for our material law of the form (\ref{eq:frac-law-1}), that is for
some constant $c_{0}\in\oi0\infty$ and all sufficiently large $\rho\in\oi0\infty$
we require
\begin{equation}
\Re\left\langle U|\partial_{0}M\left(\partial_{0}^{-1}\right)U\right\rangle _{\rho,0,0}\geq c_{0}\left\langle U|U\right\rangle _{\rho,0,0}\label{eq:solution-condition}
\end{equation}
for all $U\in D\left(\partial_{0}\right)$ by the unitarity of the
Fourier-Laplace transformation. 

\noindent If (\ref{eq:solution-condition}) holds for the case $M_{2}\left(\partial_{0}^{-1}\right)=0$
with a constant $c_{0}\in\oi0\infty$ and for all sufficiently large
$\rho\in\oi0\infty$, we may consider a non-vanishing remainder term
$M_{2}\left(\partial_{0}^{-1}\right)$ as a perturbation by assuming
\begin{equation}
\limsup_{\rho\to\infty}\left\Vert M_{2}\left(\partial_{0}^{-1}\right)\right\Vert _{H_{\rho,0}\left(\mathbb{R},H\right)}<c_{0},\label{eq:perturb-frac}
\end{equation}
yielding an estimate of the form (\ref{eq:solution-condition}) with
$c_{0}$ replaced by a positive constant $\tilde{c}_{0}<c_{0}-\limsup_{\rho\to\infty}\left\Vert M_{2}\left(\partial_{0}^{-1}\right)\right\Vert _{H_{\rho,0}\left(\mathbb{R},H\right)}$. 

\noindent Our aim is to give conditions in more explicit terms which
warrant condition (\ref{eq:solution-condition}). 

\noindent For convenience we shall use a monotonically increasing
enumeration $\left(\alpha_{0},\ldots,\alpha_{N}\right)$ of $\Pi$.
For the purpose of our following considerations we first record the
following rather elementary observation.
\begin{lem}
\label{fac-lem}Let $\iota$ be the canonical (isometric) embedding
of a closed subspace $V\subseteq H$ into $H.$ Then $\iota\iota^{*}$
is the orthogonal projector onto $V$. Let $\kappa$ be the canonical
embedding of $V^{\perp}$ into $H$ then we have
\begin{equation}
\iota\iota^{*}+\kappa\kappa^{*}=1.\label{eq:pro--sum}
\end{equation}
\end{lem}
\begin{proof}
It is 
\begin{align*}
\iota:V & \to H\\
x & \mapsto x
\end{align*}
an isometry with the inner product $\left\langle \:\cdot\:|\:\cdot\:\right\rangle _{V}$
taken as just the restriction of the inner product $\left\langle \:\cdot\:|\:\cdot\:\right\rangle _{H}$
of the Hilbert space $H$. We find 
\[
\left\langle x|y\right\rangle _{H}=\left\langle \iota x|y\right\rangle _{H}=\left\langle x|\iota^{*}y\right\rangle _{V}
\]
for all $x\in V$ and $y\in H$. We read off that for $y\in V$ we
have
\[
\left\langle x|y\right\rangle _{V}=\left\langle x|y\right\rangle _{H}=\left\langle \iota x|y\right\rangle _{H}=\left\langle x|\iota^{*}y\right\rangle _{V}
\]
for all $x\in V$ and so
\[
\iota^{*}y=y.
\]
Moreover, for $y\in V^{\perp}$ we read off
\[
0=\left\langle x|\iota^{*}y\right\rangle _{V}
\]
for all $x\in V$, i.e.
\[
\iota^{*}y=0.
\]
Consequently, we also have
\begin{align*}
\iota\iota^{*}y & =y\mbox{ for }y\in V,\\
\iota\iota^{*}y & =0\mbox{ for }y\in V^{\perp}.
\end{align*}
These properties, however, characterize the orthogonal projector onto
$V$. Similarly for $\kappa$ and the projection theorem yields (\ref{eq:pro--sum}).\end{proof}
\begin{rem}
\label{rem-block}Equality (\ref{eq:pro--sum}) may be written in
an intuitive block operator matrix notation as
\[
\left(\begin{array}{cc}
\iota & \kappa\end{array}\right)\left(\begin{array}{c}
\iota^{*}\\
\kappa^{*}
\end{array}\right)=1.
\]
It may also be worth noting that $\iota^{*}\iota:V\to V$ and $\kappa^{*}\kappa:V^{\perp}\to V^{\perp}$
are just the identities on $V$ and $V^{\perp}$, respectively. It
is common practice to identify $H$ and $V\oplus V^{\perp}$, which
makes
\begin{equation}
\left(\begin{array}{c}
\iota^{*}\\
\kappa^{*}
\end{array}\right):H\to V\oplus V^{\perp}\eqqcolon\left(\begin{array}{c}
V\\
V^{\perp}
\end{array}\right),\; x\mapsto\left(\begin{array}{c}
\iota^{*}\\
\kappa^{*}
\end{array}\right)x=\left(\begin{array}{c}
\iota^{*}x\\
\kappa^{*}x
\end{array}\right)\label{eq:proj}
\end{equation}
the identity. For our purposes it appears helpful to avoid this identification.
The mapping (\ref{eq:proj}) is obviously unitary, which allows us
for example to study an equation of the form 
\[
NU=F
\]
for a bounded linear operator $N$ in $H$ via the unitarily equivalent
block operator matrix equation 
\[
\left(\begin{array}{c}
\iota^{*}\\
\kappa^{*}
\end{array}\right)N\left(\begin{array}{cc}
\iota & \kappa\end{array}\right)\left(\begin{array}{c}
\iota^{*}\\
\kappa^{*}
\end{array}\right)U=\left(\begin{array}{cc}
\iota^{*}N\iota & \iota^{*}N\kappa\\
\kappa^{*}N\iota & \kappa^{*}N\kappa
\end{array}\right)\left(\begin{array}{c}
\iota^{*}U\\
\kappa^{*}U
\end{array}\right)=\left(\begin{array}{c}
\iota^{*}F\\
\kappa^{*}F
\end{array}\right).
\]
Note that
\[
\left(\begin{array}{cc}
\iota & \kappa\end{array}\right):\left(\begin{array}{c}
V\\
V^{\perp}
\end{array}\right)\to H,\;\left(\begin{array}{c}
u\\
v
\end{array}\right)\mapsto\left(\begin{array}{cc}
\iota & \kappa\end{array}\right)\left(\begin{array}{c}
u\\
v
\end{array}\right)=\iota u+\kappa v
\]
is the inverse of (\ref{eq:proj}).
\end{rem}
\noindent Let $P_{0}$, $F_{0}$, $Q_{0}$ orthogonal projectors in
$H$ with 
\[
P_{0}+F_{0}+Q_{0}=1
\]
and 
\[
P_{0}=\iota_{P_{0}}\iota_{P_{0}}^{*},\: F_{0}=\iota_{F_{0}}\iota_{F_{0}}^{*},\: Q_{0}=\iota_{Q_{0}}\iota_{Q_{0}}^{*}
\]
the corresponding factorizations in the above sense. We shall impose
the following general assumption.
\begin{condition}
\label{cond-frac-posdef}Let the continuous linear operator $M_{1}$
and the selfadjoint, continuous operators $M_{0}$, $M_{\alpha_{k}}$,
$k\in\{0,\ldots,N\}$, be such that $M_{0}$, $M_{\alpha_{k}}$ are
commuting with $P_{0}$, $Q_{0}$ (and so also $F_{0}$). Moreover,
let $P_{0}M_{\alpha_{k}}P_{0},\; Q_{0}M_{\alpha_{k}}Q_{0}\geq0$,
$k\in\{0,\ldots,N\}$, and $M_{0}$ non-negative and such that 
\begin{equation}
\iota_{P_{0}}^{*}M_{0}\iota_{P_{0}},\,\iota_{Q_{0}}^{*}\Re M_{1}\iota_{Q_{0}},\begin{array}{l}
\iota_{F_{0}}^{*}M_{\alpha_{0}}\iota_{F_{0}}\;\mbox{ are strictly positive definite. }\end{array}\label{eq:pos-def-frac-1}
\end{equation}
\end{condition}
\begin{thm}
Let Condition \ref{cond-frac-posdef} hold. Then the well-posedness
requirement (\ref{eq:solution-condition}) for 
\[
M(\partial_{0}^{-1})=M_{0}+\sum_{\alpha\in\Pi}\partial_{0}^{-\alpha}M_{\alpha}+\partial_{0}^{-1}M_{1}
\]
is satisfied.\end{thm}
\begin{proof}
By the Fourier-Laplace transformation, it suffices to estimate $\rho M_{0}+\sum_{k=0}^{N}\phi_{\alpha}\left(\rho\right)M_{\alpha_{k}}+\Re M_{1}$
uniformly with respect to $t\in\mathbb{R}$ and all sufficiently large
$\rho\in\oi0\infty$. Here, we have set $\Re\left(\left(\i t+\rho\right)^{1-\alpha}\right)\eqqcolon\phi_{\alpha}\left(\rho\right)$. 

Due to the assumed commutativity of $P_{0}$ and $M_{\alpha}$ and
the non-negativity of $P_{0}M_{\alpha}P_{0}=P_{0}M_{\alpha}=M_{\alpha}P_{0}$
and $Q_{0}M_{\alpha}Q_{0}=Q_{0}M_{\alpha}=M_{\alpha}Q_{0}$, $\alpha\in\Pi$,
we  get with some generic constants $c_{0},c_{1},c_{2}\in\oi0\infty$
using standard estimates of the form $2ab\leq\epsilon a^{2}+\epsilon^{-1}b^{2}$
that
\begin{equation}
\begin{array}{l}
\rho M_{0}+\sum_{k=0}^{N}\phi_{\alpha_{k}}\left(\rho\right)M_{\alpha_{k}}+\Re M_{1}\\
=\rho P_{0}M_{0}P_{0}+\rho F_{0}M_{0}F_{0}+\rho Q_{0}M_{0}Q_{0}+\sum_{k=0}^{N}\phi_{\alpha_{k}}\left(\rho\right)P_{0}M_{\alpha_{k}}P_{0}\\
\quad+\sum_{k=0}^{N}\phi_{\alpha_{k}}\left(\rho\right)F_{0}M_{\alpha_{k}}F_{0}+\sum_{k=0}^{N}\phi_{\alpha_{k}}\left(\rho\right)Q_{0}M_{\alpha_{k}}Q_{0}\\
\quad+P_{0}\Re M_{1}P_{0}+F_{0}\Re M_{1}F_{0}+Q_{0}\Re M_{1}Q_{0}\\
\quad+2\Re F_{0}\Re M_{1}P_{0}+2\Re Q_{0}\Re M_{1}P_{0}+2\Re Q_{0}\Re M_{1}F_{0}\\
\geq\rho P_{0}M_{0}P_{0}+\sum_{k=0}^{N}\phi_{\alpha_{k}}\left(\rho\right)F_{0}M_{\alpha_{k}}F_{0}\\
\quad+P_{0}\Re M_{1}P_{0}+F_{0}\Re M_{1}F_{0}+Q_{0}\Re M_{1}Q_{0}\\
\quad+2\Re F_{0}\Re M_{1}P_{0}+2\Re Q_{0}\Re M_{1}P_{0}+2\Re Q_{0}\Re M_{1}F_{0}\\
\geq\rho c_{0}P_{0}+\sum_{k=0}^{N}\phi_{\alpha_{k}}\left(\rho\right)F_{0}M_{\alpha_{k}}F_{0}-c_{2}F_{0}+c_{1}Q_{0}
\end{array}\label{eq:estim-mat}
\end{equation}
holds for all sufficiently large $\rho\in\oi0\infty$. In consequence
we may reduce our considerations to 
\[
\sum_{k=0}^{N}\phi_{\alpha_{k}}\left(\rho\right)F_{0}M_{\alpha_{k}}F_{0}-c_{2}F_{0}.
\]
Indeed, using (\ref{eq:monot}) we obtain 

\begin{equation}
\begin{array}{l}
\sum_{k=0}^{N}\phi_{\alpha_{k}}\left(\rho\right)\iota_{F_{0}}^{*}M_{\alpha_{k}}\iota_{F_{0}}\geq\\
\geq\phi_{\alpha_{0}}\left(\rho\right)c_{3}+\sum_{k=1}^{N}\phi_{\alpha_{k}}\left(\rho\right)\iota_{F_{0}}^{*}M_{\alpha_{k}}\iota_{F_{0}}\\
\geq c_{3}\phi_{\alpha_{0}}\left(\rho\right)-c_{4}\phi_{\alpha_{1}}\left(\rho\right),
\end{array}\label{eq:F-Est}
\end{equation}
for some $c_{3},c_{4}>0.$ Thus, from (\ref{eq:F-Est}) we get for
all sufficiently large $\rho\in\oi0\infty$ 
\[
\sum_{k=0}^{N}\phi_{\alpha_{k}}\iota_{F_{0}}^{*}M_{\alpha_{k}}\iota_{F_{0}}\geq\frac{c_{3}}{2}\phi_{\alpha_{0}}\left(\rho\right),
\]
where we use that 
\[
\frac{\phi_{\alpha_{1}}(\rho)}{\phi_{\alpha_{0}}(\rho)}\to0\quad(\rho\to\infty)
\]
by (\ref{eq:real_part_frac}). Using that according to Lemma \ref{lem-frac-est-1}
we have $\phi_{\alpha_{0}}\left(\rho\right)\geq\rho^{1-\alpha_{0}}$
and inserting this into (\ref{eq:estim-mat}) yields 
\begin{align*}
\rho M_{0}+\sum_{k=0}^{N}\phi_{\alpha_{k}}\left(\rho\right)M_{\alpha_{k}}+\Re M_{1} & \geq\rho c_{0}P_{0}+\left(\frac{c_{3}}{2}\rho^{1-\alpha_{0}}-c_{1}\right)F_{0}+c_{2}Q_{0}\\
 & \geq\min\left\{ \rho c_{0},\frac{c_{3}}{2}\rho^{1-\alpha_{0}}-c_{1},c_{2}\right\} 
\end{align*}
and so the desired strict positive definiteness holds indeed for all
sufficiently large $\rho\in\oi0\infty$.
\end{proof}
We may summarize our findings in the following well-posedness theorem,
which according to the above is just a particular case of the main
result of \cite{Pi2009-1} (see Theorem \ref{thm:sol_theory_orig}).
\begin{thm}
\label{Thm:Sol}Let \textup{$M\left(\partial_{0}^{-1}\right)$ }\textup{\emph{be
a material law of the form (\ref{eq:frac-law}) satisfying Condition
\ref{cond-frac-posdef} and (\ref{eq:perturb-frac}). Then for every
sufficiently large $\rho\in\mathbb{R}_{>0}$ and every $f\in H_{\rho,k}\left(\mathbb{R},H\right)$
there is a unique solution $U\in H_{\rho,k}\left(\mathbb{R},H\right)$,
$k\in\mathbb{Z},$ satisfying}} equation (\ref{eq:evo_2}). Moreover,
the solution operator 
\[
\left(\partial_{0}M\left(\partial_{0}^{-1}\right)+A\right)^{-1}:H_{\rho,k}\left(\mathbb{R},H\right)\to H_{\rho,k}\left(\mathbb{R},H\right)
\]
is continuous and causal in the sense that, for every $a\in\mathbb{R}$,
if $g\in H_{\rho,k}\left(\mathbb{R},H\right)$ vanishes as an $H$-valued
generalized function on $]-\infty,a[$ then so does $\left(\partial_{0}M\left(\partial_{0}^{-1}\right)+A\right)^{-1}g$.\end{thm}
\begin{rem}
~\end{rem}
\begin{enumerate}
\item Causality implies that if $g\in H_{\rho,0}\left(\mathbb{R},H\right)$
vanishes on $]-\infty,a[$ then the fractional derivatives occurring
in the material law are actually Riemann-Liouville type derivatives
$_{a}D_{t}^{\alpha_{k}}$, $k\in\{0,\ldots,N\}.$ 
\item We have chosen to discuss the case of $A$ skew-selfadjoint in $H.$
More general operators $A$ could also be admitted by using other
suitable boundary conditions, compare e.g. \cite{Picard2010,Trostorff2012_nonlin_bd}.
The operator $A$ can even be generalized to a monotone relation,
see \cite{Trostorff2012_nonlin_bd}.
\end{enumerate}
We shall give a first illustrative example in which case the above
theorem may be applied.
\begin{example}
Let $H_{0},H_{1},H_{2}$ be Hilbert spaces. For $\alpha,\beta\in]0,1[$
with $\alpha<\beta$ we consider the following operator matrices with
respect to the Hilbert space $H_{0}\oplus H_{1}\oplus H_{2}$:
\[
M_{0}\coloneqq\left(\begin{array}{ccc}
1 & 0 & 0\\
0 & 1 & 0\\
0 & 0 & 0
\end{array}\right),\quad M_{\alpha}\coloneqq\left(\begin{array}{ccc}
0 & 0 & 0\\
0 & 1 & 0\\
0 & 0 & 0
\end{array}\right),\quad M_{\beta}\coloneqq\left(\begin{array}{ccc}
0 & 0 & 0\\
0 & B & 0\\
0 & 0 & 0
\end{array}\right),\quad M_{1}\coloneqq\left(\begin{array}{ccc}
0 & 0 & 0\\
0 & 0 & 0\\
0 & 0 & 1
\end{array}\right).
\]
 For some selfadjoint operator $B\in L(H_{1})$. The resulting material
law is given by
\[
M_{0}+\partial_{0}^{-\alpha}M_{\alpha}+\partial_{0}^{-\beta}M_{\beta}+\partial_{0}^{-1}M_{1}.
\]
In the situation of Condition \ref{cond-frac-posdef}, we have $P_{0}=\left(\begin{array}{ccc}
1 & 0 & 0\\
0 & 0 & 0\\
0 & 0 & 0
\end{array}\right)$, $F_{0}=\left(\begin{array}{ccc}
0 & 0 & 0\\
0 & 1 & 0\\
0 & 0 & 0
\end{array}\right)$ and $Q_{0}=\left(\begin{array}{ccc}
0 & 0 & 0\\
0 & 0 & 0\\
0 & 0 & 1
\end{array}\right)$. Then it is easy to see that Condition \ref{cond-frac-posdef} is
satisfied. Furthermore, note that since one has to compensate the
coefficient $M_{\beta}$, one cannot choose $P_{0}=\left(\begin{array}{ccc}
1 & 0 & 0\\
0 & 1 & 0\\
0 & 0 & 0
\end{array}\right)$, $F_{0}=\left(\begin{array}{ccc}
0 & 0 & 0\\
0 & 0 & 0\\
0 & 0 & 0
\end{array}\right)$. 
\end{example}

\section{Some Applications\label{sec:Some-Applications}}

We want first  to focus on the case $M_{0}=0.$ The most simple case
of this type is now

\[
\Pi=\left\{ \alpha\right\} 
\]
 with $\alpha=1-\beta,$ $\beta\in]0,1[$, yielding problems of the
form 
\begin{equation}
\left(\partial_{0}^{\beta}M_{\alpha}+M_{1}+A\right)U=F,\label{eq:simple-case}
\end{equation}
for given bounded linear operators $M_{\alpha}$, $M_{1}$, a skew-selfadjoint
operator $A$ in a Hilbert space $H$, right-hand side $F\in H_{\rho,0}(\mathbb{R},H)$
and suitable $\rho\in\oi0\infty$. A particular instance of this first
problem class is the fractional Fokker-Planck equation to be discussed
in Section \ref{sub:Fractional-Reaction-Diffusion-Eq}. A more complicated
material law including an infinite number of fractional integrals
is discussed in Section \ref{sec:An-Application-to} within the context
of fractional visco-elasticity. Section \ref{sub:Initial-Boundary-Value}
outlines possible ways to formulate initial (boundary) value problems
for evolutionary equations of the form (\ref{eq:simple-case}).

\subsection{Fractional Reaction-Diffusion Equations\label{sub:Fractional-Reaction-Diffusion-Eq}}

Let us consider the fractional Fokker-Planck equation, compare e.g.\ \cite{metzler1999anomalous}.
The operators involved can be defined as follows. Let $\Omega\subseteq\mathbb{R}^{n}$
be an open subset. We denote
\begin{align*}
\grad_{c}\colon\interior C_{\infty}\left(\Omega\right) & \subseteq L^{2}(\Omega)\to L^{2}(\Omega)^{n}\\
f & \mapsto\left(\partial_{i}f\right)_{i\in\{1,\ldots,n\}}
\end{align*}
and $\dive\coloneqq-\grad_{c}^{*}$ and $\interior\grad\coloneqq\overline{\grad_{c}}$%
\footnote{Of course the domain of $\interior\grad$ is the Sobolev space $H_{0}^{1}(\Omega)$
of weakly differentiable functions satisfying generalized homogeneous
Dirichlet boundary conditions. The domain of $\dive$ is the set of
$L^{2}(\Omega)$-vector fields with distributional divergence in $L^{2}(\Omega)$.%
}. The operator matrix 
\[
A\coloneqq\left(\begin{array}{cc}
0 & \dive\\
\interior\grad & 0
\end{array}\right)
\]
is then skew-selfadjoint. Of course other boundary conditions can
also be treated, see e.g.\ \cite{Picard2010,Trostorff2012_nonlin_bd,PDE_DeGruyter}.
Let $\alpha\in]0,1[$. For the material we assume
\[
M_{\alpha}\coloneqq\left(\begin{array}{cc}
\kappa_{\alpha} & 0\\
0 & 0
\end{array}\right),\: M_{1}\coloneqq\left(\begin{array}{cc}
\mu_{00} & \mu_{01}\\
\mu_{10} & \mu_{11}
\end{array}\right)
\]
with $\kappa_{\alpha},\:\Re\mu_{11}$ continuous, selfadjoint and
strictly positive definite on the respective $L^{2}$-type block component
spaces, $\mu_{ik}$ bounded, linear operators, $i,k\in\{0,1\}$. Thus,
the block operator matrix yielding the (fractional) Fokker-Planck
equation looks formally like
\[
\left(\begin{array}{cc}
\partial_{0}^{1-\alpha}\kappa_{\alpha}+\mu_{00} & \quad\dive+\mu_{01}\\
\interior\grad+\mu_{10} & \quad\mu_{11}
\end{array}\right).
\]
By our general theory developed above, the latter operator matrix
is continuously invertible in the Hilbert space of $H_{\rho,0}$-functions
taking values in $L^{2}(\Omega)$ for sufficiently large $\rho>0$.
A simple row operation yields
\[
\left(\begin{array}{cc}
\partial_{0}^{1-\alpha}\kappa_{\alpha}+\mu_{00}-\left(\dive+\mu_{01}\right)\mu_{11}^{-1}\left(\interior\grad+\mu_{10}\right) & \quad0\\
\interior\grad+\mu_{10} & \quad\mu_{11}
\end{array}\right).
\]
The operator in the first row of the block operator matrix reduces
to the following fractional convection-diffusion operator 
\begin{align*}
 & \partial_{0}^{1-\alpha}\kappa_{\alpha}+\mu_{00}-\left(\dive+\mu_{01}\right)\mu_{11}^{-1}\left(\interior\grad+\mu_{10}\right)\\
 & =\partial_{0}^{1-\alpha}\kappa_{\alpha}+\left(\mu_{00}-\mu_{01}\mu_{11}^{-1}\mu_{10}\right)-\dive\mu_{11}^{-1}\interior\grad-\mu_{01}\mu_{11}^{-1}\interior\grad-\dive\mu_{11}^{-1}\mu_{10}.
\end{align*}
 The second row represents the flux $\Phi$ in terms of the unknown
probability density $\theta$: 
\begin{equation}
\Phi=-\mu_{11}^{-1}\interior\grad\theta-\mu_{11}^{-1}\mu_{10}\theta.\label{eq:Fourier}
\end{equation}
Clearly, formally for $\alpha=0$ we recover the usual convection-diffusion
equation. Speaking in terms of heat conduction, which is, up to a
choice of units (and interpretation), governed by the same type of
convection-diffusion equation, (\ref{eq:Fourier}) is a variant of
Fourier's law.

\subsection{\label{sec:An-Application-to}An Application to a General Fractional
Kelvin-Voigt Model for Visco-Elastic Solids.}

As another application we consider the material law associated with
a fractional Kelvin-Voigt model. It is a well-established idea that
material laws describing visco-elastic solids may be more simple to
match with measurements if fractional integrals are admitted, see
e.g. \cite{pre05172699}, \cite{doi:10.1142/S0218396X03002024}. %
{} The spatial operator
\[
A\coloneqq\left(\begin{array}{cc}
0 & -\Div\\
-\Grad & 0
\end{array}\right),
\]
where $\Div$ is the restriction of the tensorial divergence operator
$\dive$ to symmetric tensors of order 2 and $\Grad$ is the symmetric
part of the Jacobian matrix $d\otimes u$ of $u$, is skew-selfadjoint
if suitable boundary conditions at the boundary of the underlying
domain $\Omega\subseteq\mathbb{R}^{3}$ are imposed, e.g., similar
to the previous example, vanishing of the displacement $u$ on the
boundary. This choice of boundary condition would amount to replacing
$\Grad$ by the closure $\interior\Grad$ of the restriction of $\Grad$
to vector fields with smooth components vanishing outside of a compact
subset of $\Omega.$ We will not be explicit on the boundary conditions
under consideration and simply assume that $A$ is considered with
a domain such that it becomes skew-selfadjoint. We refer to \cite[Section 4.2]{Trostorff2012_nonlin_bd}
or \cite[Section 2]{waurick2013_fract_elast} for more details. The
above general perspective to evolutionary equations allow for very
general materials to be covered by the approach.

Let us take a closer look at a fractional Kelvin-Voigt model for visco-elastic
material. In this model, given an external forcing term $f$, we have
the equation 
\begin{align*}
\partial_{0}\eta\partial_{0}u-\mathrm{Div}T & =f
\end{align*}
linking the stress tensor field $T$ with the displacement vector
field $u$, where $\eta$ is a bounded selfadjoint and strictly positive
definite operator, accompanied by a material relation of the form
\begin{equation}
T=\left(C+D\partial_{0}^{\alpha}\right)\mathcal{E}\label{eq:KV}
\end{equation}
with $\mathcal{E}\coloneqq\Grad u$, $\alpha\in\rci01$. Here the
case $D=0$ would correspond to purely elastic behavior and the case
$\alpha=1$ and $D\not=0$ would describe the classical Kelvin-Voigt
material. Introducing $v\coloneqq\partial_{0}u$ as a new unknown
and differentiating (\ref{eq:KV}), we obtain 
\begin{align*}
\partial_{0}\eta v-\mathrm{Div}T & =f\\
\partial_{0}\left(C+D\partial_{0}^{\alpha}\right)^{-1}T & =\mathrm{Grad}v
\end{align*}
yielding formally an evolutionary equation of the form
\begin{align*}
\left(\partial_{0}\left(\begin{array}{cc}
\eta & 0\\
0 & \left(C+D\partial_{0}^{\alpha}\right)^{-1}
\end{array}\right)+\left(\begin{array}{cc}
0 & -\Div\\
-\Grad & 0
\end{array}\right)\right)\left(\begin{array}{c}
v\\
T
\end{array}\right) & =\left(\begin{array}{c}
f\\
0
\end{array}\right).
\end{align*}

Thus, we obtain a system of the form
\begin{align*}
\left(\partial_{0}M\left(\partial_{0}^{-1}\right)+A\right)\left(\begin{array}{c}
v\\
T
\end{array}\right) & =\left(\begin{array}{c}
f\\
0
\end{array}\right)
\end{align*}
where $A$ is a skew-selfadjoint realization of the formal block operator
matrix given by
\[
\left(\begin{array}{cc}
0 & -\Div\\
-\Grad & 0
\end{array}\right)
\]
and the material law operator 
\[
M\left(\partial_{0}^{-1}\right)\coloneqq\left(\begin{array}{cc}
\eta & 0\\
0 & \left(C+D\partial_{0}^{\alpha}\right)^{-1}
\end{array}\right).
\]
In order to warrant existence and boundedness of the operator $\left(C+D\partial_{0}^{\alpha}\right)^{-1}$,
we impose the condition %
{} $C,D\geq0$ and
\begin{equation}
C+D\rho^{\alpha}\geq c_{0}\label{eq:posdef-KV}
\end{equation}
for some $c_{0}\in\mathbb{R}_{>0}$ and all sufficiently large $\rho\in\oi0\infty$.
However, in order to show that $M\left(\partial_{0}^{-1}\right)$
satisfies the well-posedness condition, we have to impose a slightly
stronger condition. 
\begin{thm}
Assume that $C,D$ are non-negative selfadjoint operators and such
that $\iota_{0,D}^{*}C\iota_{0,D}$, $\iota_{1,D}^{*}D\iota_{1,D}$
are both strictly positive definite, where $\iota_{0,D}$ and $\iota_{1,D}$
are the canonical embeddings of the null space and range of $D$ in
the underlying Hilbert space of symmetric $3\times3$-matrices with
entries in $L^{2}\left(\Omega\right)$, respectively. Then estimate
(\ref{eq:posdef-KV}) holds and \textup{$\left(C+D\partial_{0}^{\alpha}\right)^{-1}$}
satisfies Condition \ref{cond-frac-posdef}.\end{thm}
\begin{proof}
First of all note that estimate (\ref{eq:posdef-KV}) is a straight
forward consequence of Euklid's inequality. Now, we have by decomposition
(see also Lemma \ref{fac-lem}):{\footnotesize{} 
\[
\left(C+D\partial_{0}^{\alpha}\right)^{-1}=\left(\begin{array}{cc}
\iota_{0,D} & \iota_{1,D}\end{array}\right)\left(\begin{array}{cc}
\iota_{0,D}^{*}C\iota_{0,D} & \iota_{0,D}^{*}C\iota_{1,D}\\
\iota_{1,D}^{*}C\iota_{0,D} & \left(\iota_{1,D}^{*}C\iota_{1,D}+\iota_{1,D}^{*}D\iota_{1,D}\partial_{0}^{\alpha}\right)
\end{array}\right)^{-1}\left(\begin{array}{c}
\iota_{0,D}^{*}\\
\iota_{1,D}^{*}
\end{array}\right).
\]
}We abbreviate $C_{jk}\coloneqq\iota_{j,D}^{*}C\iota_{k,D}$ $(j,k\in\{0,1\})$,
$D_{11}\coloneqq\iota_{1,D}^{*}D\iota_{1,D}$, 
\[
\tilde{C}_{11}\coloneqq C_{11}-C_{10}C_{00}^{-1}C_{01}
\]
and $W\coloneqq\left(\begin{array}{cc}
1 & -C_{00}^{-1}C_{01}\\
0 & 1
\end{array}\right)$. Then we get that
\[
\begin{array}{l}
\left(\begin{array}{cc}
C_{00} & C_{01}\\
C_{10} & \left(C_{11}+D_{11}\partial_{0}^{\alpha}\right)
\end{array}\right)^{-1}=W\left(\begin{array}{cc}
C_{00}^{-1} & 0\\
0 & \left(\tilde{C}_{11}+D_{11}\partial_{0}^{\alpha}\right)^{-1}
\end{array}\right)W^{*}.\end{array}
\]
With $\mathcal{G}\left(\partial_{0}^{-1}\right)\coloneqq\left(\partial_{0}^{-\alpha}\sqrt{D_{11}^{-1}}\tilde{C}_{11}\sqrt{D_{11}^{-1}}+1\right)^{-1}=\sum_{n=0}^{\infty}\partial_{0}^{-n\alpha}\left(-\sqrt{D_{11}^{-1}}\tilde{C}_{11}\sqrt{D_{11}^{-1}}\right)^{n}$
for $\rho$ sufficiently large, we have
\[
\begin{array}{l}
\left(\tilde{C}_{11}+D_{11}\partial_{0}^{\alpha}\right)^{-1}=\partial_{0}^{-\alpha}\sqrt{D_{11}^{-1}}\mathcal{G}\left(\partial_{0}^{-1}\right)\sqrt{D_{11}^{-1}}\end{array}.
\]
Thus,
\begin{align*}
 & \left(C+D\partial_{0}^{\alpha}\right)^{-1}\\
 & =\left(\begin{array}{cc}
\iota_{0,D} & \iota_{1,D}\end{array}\right)W\left(\begin{array}{cc}
C_{00}^{-1} & 0\\
0 & 0
\end{array}\right)W^{*}\left(\begin{array}{c}
\iota_{0,D}^{*}\\
\iota_{1,D}^{*}
\end{array}\right)+\\
 & \quad\left(\begin{array}{cc}
\iota_{0,D} & \iota_{1,D}\end{array}\right)W\left(\begin{array}{cc}
0 & 0\\
0 & \partial_{0}^{-\alpha}\sqrt{D_{11}^{-1}}\mathcal{G}\left(\partial_{0}^{-1}\right)\sqrt{D_{11}^{-1}}
\end{array}\right)W^{*}\left(\begin{array}{c}
\iota_{0,D}^{*}\\
\iota_{1,D}^{*}
\end{array}\right)\\
 & =\left(\begin{array}{cc}
\iota_{0,D} & \iota_{1,D}\end{array}\right)W\left(\begin{array}{cc}
C_{00}^{-1} & 0\\
0 & 0
\end{array}\right)W^{*}\left(\begin{array}{c}
\iota_{0,D}^{*}\\
\iota_{1,D}^{*}
\end{array}\right)\\
 & \quad+\sum_{n=0}^{\infty}\left(\begin{array}{cc}
\iota_{0,D} & \iota_{1,D}\end{array}\right)W\left(\begin{array}{cc}
0 & 0\\
0 & \partial_{0}^{-\alpha}\sqrt{D_{11}^{-1}}\partial_{0}^{-n\alpha}\left(-\sqrt{D_{11}^{-1}}\tilde{C}_{11}\sqrt{D_{11}^{-1}}\right)^{n}\sqrt{D_{11}^{-1}}
\end{array}\right)W^{*}\left(\begin{array}{c}
\iota_{0,D}^{*}\\
\iota_{1,D}^{*}
\end{array}\right)\\
 & =\left(\begin{array}{cc}
\iota_{0,D} & \iota_{1,D}\end{array}\right)W\left(\left(\begin{array}{cc}
C_{00}^{-1} & 0\\
0 & 0
\end{array}\right)\right)W^{*}\left(\begin{array}{c}
\iota_{0,D}^{*}\\
\iota_{1,D}^{*}
\end{array}\right)\\
 & \quad+\sum_{n=0}^{\lceil\alpha^{-1}\rceil-1}\left(\begin{array}{cc}
\iota_{0,D} & \iota_{1,D}\end{array}\right)W\left(\begin{array}{cc}
0 & 0\\
0 & \partial_{0}^{-\alpha}\sqrt{D_{11}^{-1}}\partial_{0}^{-n\alpha}\left(-\sqrt{D_{11}^{-1}}\tilde{C}_{11}\sqrt{D_{11}^{-1}}\right)^{n}\sqrt{D_{11}^{-1}}
\end{array}\right)W^{*}\left(\begin{array}{c}
\iota_{0,D}^{*}\\
\iota_{1,D}^{*}
\end{array}\right)\\
 & \quad+\sum_{n=\lceil\alpha^{-1}\rceil}^{\infty}\left(\begin{array}{cc}
\iota_{0,D} & \iota_{1,D}\end{array}\right)W\left(\begin{array}{cc}
0 & 0\\
0 & \partial_{0}^{-\alpha}\sqrt{D_{11}^{-1}}\partial_{0}^{-n\alpha}\left(-\sqrt{D_{11}^{-1}}\tilde{C}_{11}\sqrt{D_{11}^{-1}}\right)^{n}\sqrt{D_{11}^{-1}}
\end{array}\right)W^{*}\left(\begin{array}{c}
\iota_{0,D}^{*}\\
\iota_{1,D}^{*}
\end{array}\right).
\end{align*}
Hence, with 
\[
M_{0}=\left(\begin{array}{cc}
\eta & 0\\
0 & \left(\begin{array}{cc}
\iota_{0,D} & \iota_{1,D}\end{array}\right)W\left(\begin{array}{cc}
C_{00}^{-1} & 0\\
0 & 0
\end{array}\right)W^{*}\left(\begin{array}{c}
\iota_{0,D}^{*}\\
\iota_{1,D}^{*}
\end{array}\right)
\end{array}\right),
\]
and 
\[
M_{\gamma}=\left(\begin{array}{cc}
0 & 0\\
0 & \left(\begin{array}{cc}
\iota_{0,D} & \iota_{1,D}\end{array}\right)W\left(\begin{array}{cc}
0 & 0\\
0 & \sqrt{D_{11}^{-1}}\left(-\sqrt{D_{11}^{-1}}\tilde{C}_{11}\sqrt{D_{11}^{-1}}\right)^{n}\sqrt{D_{11}^{-1}}^{-1}
\end{array}\right)W^{*}\left(\begin{array}{c}
\iota_{0,D}^{*}\\
\iota_{1,D}^{*}
\end{array}\right)
\end{array}\right)
\]
 for $\gamma\in\Pi\coloneqq\{\left(1+n\right)\alpha;n\in\{0,\ldots,\lceil\alpha^{-1}\rceil-1\}\}$
we have
\[
M\left(\partial_{0}^{-1}\right)=M_{0}+\sum_{\gamma\in\Pi}\partial_{0}^{-\gamma}M_{\gamma}+\partial_{0}^{-1}M_{2}\left(\partial_{0}^{-1}\right),
\]
where
\begin{align*}
 & \partial_{0}^{-1}M_{2}\left(\partial_{0}^{-1}\right)\\
 & \coloneqq\sum_{n=\lceil\alpha^{-1}\rceil}^{\infty}\left(\begin{array}{cc}
\iota_{0,D} & \iota_{1,D}\end{array}\right)W\left(\begin{array}{cc}
0 & 0\\
0 & \partial_{0}^{-\alpha}\sqrt{D_{11}^{-1}}\partial_{0}^{-n\alpha}\left(-\sqrt{D_{11}^{-1}}\tilde{C}_{11}\sqrt{D_{11}^{-1}}\right)^{n}\sqrt{D_{11}^{-1}}
\end{array}\right)W^{*}\left(\begin{array}{c}
\iota_{0,D}^{*}\\
\iota_{1,D}^{*}
\end{array}\right).
\end{align*}
Now, for $\rho\in\mathbb{R}_{>0}$ sufficiently large we estimate
\begin{align*}
\left\Vert M_{2}\left(\partial_{0}^{-1}\right)\right\Vert  & \leq K_{0}\left\Vert \partial_{0}\sum_{n=\lceil\alpha^{-1}\rceil}^{\infty}\partial_{0}^{-\alpha}\partial_{0}^{-n\alpha}\left(-\sqrt{D_{11}^{-1}}\tilde{C}_{11}\sqrt{D_{11}^{-1}}\right)^{n}\right\Vert \\
 & \leq K_{0}\sum_{n=\lceil\alpha^{-1}\rceil}^{\infty}\rho^{-\alpha+1-n\alpha}K_{1}^{n}=K_{0}\rho^{1-\alpha}\sum_{n=\lceil\alpha^{-1}\rceil}^{\infty}\left(\rho^{-\alpha}K_{1}\right)^{n}\\
 & =K_{0}\rho^{1-\alpha}\left(\frac{1}{1-\rho^{-\alpha}K_{1}}-\frac{1-\rho^{-\lceil\alpha^{-1}\rceil\alpha}K_{1}^{\lceil\alpha^{-1}\rceil}}{1-\rho^{-\alpha}K_{1}}\right)\\
 & =K_{0}\rho^{1-\alpha}\left(\frac{\rho^{-\lceil\alpha^{-1}\rceil\alpha}K_{1}^{\lceil\alpha^{-1}\rceil}}{1-\rho^{-\alpha}K_{1}}\right)=K_{0}\rho^{1-\alpha-\lceil\alpha^{-1}\rceil\alpha}\left(\frac{K_{1}^{\lceil\alpha^{-1}\rceil}}{1-\rho^{-\alpha}K_{1}}\right)\\
 & \leq K_{0}\rho^{-\alpha}\left(\frac{K_{1}^{\lceil\alpha^{-1}\rceil}}{1-\rho^{-\alpha}K_{1}}\right)\to0\quad(\rho\to\infty),
\end{align*}
where $K_{0}\coloneqq\left\Vert W\sqrt{D_{11}^{-1}}\right\Vert ^{2}$
and $K_{1}\coloneqq\left\Vert \sqrt{D_{11}^{-1}}\tilde{C}_{11}\sqrt{D_{11}^{-1}}\right\Vert $.
Thus, Condition \ref{cond-frac-posdef} is satisfied, for the respective
choices %
\[
P_{0}=\left(\begin{array}{cc}
1 & 0\\
0 & 0
\end{array}\right),\quad F_{0}=\left(\begin{array}{cc}
0 & 0\\
0 & 1
\end{array}\right),\text{ and }Q_{0}=0.\tag*{\qedhere}
\]
\end{proof}
\begin{rem}
In \cite{waurick2013_fract_elast} a well-posedness result for the
fractional Kelvin-Voigt model considered was shown under the condition
that $\alpha\geqq1/2$.
\end{rem}

\subsection{Initial Boundary Value Problems\label{sub:Initial-Boundary-Value}}

We now consider initial value problems for equations of the form (\ref{eq:simple-case}).
Since fractional derivatives of order $\beta\in]0,1[$ are non-local,
i.e., we have memory effects, prescribing a pre-history appears to
be more appropriate. This amounts to imposing $H_{\rho,0}\left(\mathbb{R},H\right)$-data
prior to the initial time $0$, i.e., ``the pre-history'', and is
clearly covered by the above.

However, if initial conditions are deemed appropriate, then there
appear to be two possible ways of implementing them. The first option
would be to consider with $F\in H_{\rho,0}\left(\mathbb{R},H\right)$,
$W_{\alpha}\in M_{\alpha}\left[H\right],$
\begin{equation}
\partial_{0}^{\beta}M_{\alpha}U+M_{1}U+AU=F+\delta\otimes W_{\alpha}\label{ibvp-1}
\end{equation}
as a straight-forward generalization of the case $\beta=1,$ see \cite[Theorem 6.2.9]{PDE_DeGruyter}.
This translates to
\[
\partial_{0}\partial_{0}^{-\alpha}M_{\alpha}U+M_{1}U+AU=F+\delta\otimes W_{\alpha},
\]
which following the rationale in \cite[Proof of Theorem 6.2.9]{PDE_DeGruyter}
would be in point-wise terms an implementation of the initial jump
condition 
\[
\left(\partial_{0}^{-\alpha}M_{\alpha}U-\chi_{_{\mathbb{R}_{>0}}}\otimes W_{\alpha}\right)\left(0+\right)-\left(\partial_{0}^{-\alpha}M_{\alpha}U-\chi_{_{\mathbb{R}_{>0}}}\otimes W_{\alpha}\right)\left(0-\right)=0,
\]
where the limits are taken in $H_{-1}\coloneqq H_{-1}(A+1)$, the
completion of $H$ with respect to the norm $|(A+1)^{-1}\cdot|_{H}$,
see also Section \ref{sub:Sobolev-chains}. If $F=0$ on $\mathbb{R}_{<0}$
then causality yields 
\[
\left(\partial_{0}^{-\alpha}M_{\alpha}U-\chi_{_{\mathbb{R}_{>0}}}\otimes W_{\alpha}\right)\left(0-\right)=0\mbox{ in }H_{-1}
\]
and so, in this case an initial condition of the form
\[
\left(\partial_{0}^{-\alpha}M_{\alpha}U\right)\left(0+\right)=W_{\alpha}\mbox{ in }H_{-1}
\]
results. 

Alternatively, let
\[
V_{\alpha}\in M_{\alpha}\left[H\right]
\]
and $G\in H_{\rho,0}\left(\mathbb{R},H\right)$ be given and consider
\begin{align}
\partial_{0}^{\beta}\left(M_{\alpha}V-\chi_{_{\mathbb{R}_{>0}}}\otimes V_{\alpha}\right)+M_{1}V+AV & =G.\label{ibvp-2}
\end{align}
For the case $\alpha=0,$ $\beta=1,$ $G=F$ and $V_{0}=W_{0}$ this
would be a re-formulation of (\ref{ibvp-1}). This justifies taking
(\ref{ibvp-2}) as an appropriate alternative generalization. Problem
(\ref{ibvp-2}) is equivalent to the evolutionary problem
\begin{align*}
\partial_{0}^{\beta}M_{\alpha}V+M_{1}V+AV & =G+\partial_{0}^{\beta}\chi_{_{\mathbb{R}_{>0}}}\otimes V_{\alpha},\\
 & =G+\partial_{0}^{-\alpha}\delta\otimes V_{\alpha}.
\end{align*}
Both approaches are clearly closely related. Indeed, for $G=\partial_{0}^{-\alpha}F$
and $V_{\alpha}=W_{\alpha}$ we see by comparison that
\[
V=\partial_{0}^{-\alpha}U.
\]

\end{document}